\documentclass{elsarticle}

\usepackage{amsmath,amssymb,amsfonts,amscd}
\usepackage[all]{xy}
\usepackage{xcolor}

\newtheorem{theorem}{Theorem}[section]
\newtheorem{definition}[theorem]{Definition}
\newtheorem{lemma}[theorem]{Lemma}
\newtheorem{corollary}[theorem]{Corollary}

\newtheorem{example}[theorem]{Example}
\newtheorem{remark}[theorem]{Remark}
\newproof{proof}{\it Proof}

\numberwithin{equation}{section}

\begin{document}

\begin{frontmatter}



\title{Cameron--Storvick theorem   associated with Gaussian paths on function space}

\author{Jae Gil Choi}
\ead{jgchoi@dankook.ac.kr}

\address{School of General Education,  
                 Dankook University,
                 Cheonan 31116, 
                 Republic of Korea}

 
\begin{abstract}
The purpose of this paper is to provide a more general Cameron--Storvick theorem
for the generalized analytic Feynman integral associated with  Gaussian process  
$\mathcal Z_k$ on a very general Wiener space $C_{a,b}[0,T]$.
The general Wiener space  $C_{a,b}[0,T]$ can be considered as the set of all
continuous sample paths of the  generalized Brownian motion process determined  
by continuous functions $a(t)$ and $b(t)$ on $[0,T]$.
As an interesting application, we apply this theorem
to evaluate the generalized analytic Feynman integral of certain monomials
in terms of Paley--Wiener--Zygmund stochastic integrals.
\end{abstract}

\begin{keyword}
Cameron--Storvick theorem \sep 
generalized analytic Feynman integral \sep
Gaussian process \sep
generalized Brownian motion process \sep
Paley--Wiener--Zygmund stochastic integral. 

\vspace{.3cm}
\MSC[2010] Primary   46G12 \sep 60G15;  Secondary   28C20 \sep 46B09    

\end{keyword}

\end{frontmatter}


  
\setcounter{equation}{0}
\section{Introduction}\label{sec:intro}

\par 
Let $(C_0[0,T],\mathcal{W},\mathfrak{m})$ denote  the  classical  Wiener space,
where $C_0[0,T]$ is  the set of  all  $\mathbb R$-valued continuous functions $x$ on $[0,T]$ with $x(0)=0$, 
$\mathcal{W}$ denotes the complete $\sigma$-field of all Wiener measurable subsets of 
$C_0[0,T]$, and  $\mathfrak{m}$ denotes  the  Wiener measure characterized by 
\[
\mathfrak{m}(\{x:x(t)\le \tau \})=\frac{1}{\sqrt{2\pi t}}\int_{-\infty}^{\tau}\exp\bigg[-\frac{u^2}{2t}\bigg]du.
\]
Using the Kolmogorov's extension theorem (for instance, see \cite{Kuo06,Yeh73}),
the Wiener space $C_0[0,T]$ can be illustrated as the set of all (continuous) sample paths 
of the standard Brownian motion process (SBMP).
In \cite{Cameron51}, Cameron provided an integration by parts formula for functionals on 
the classical Wiener space $C_0[0,T]$. 
In \cite{CS91}, Cameron and Storvick developed the  parts formula
for the analytic Feynman integral of functionals on $C_0[0,T]$. 
They also applied their result to establish the evaluation formula for  the analytic Feynman 
integral of unbounded functionals on $C_0[0,T]$.
The  parts formula on $C_0[0,T]$  introduced in \cite{Cameron51} also
was developed in \cite{PS98-PanAmer,PSS98-RCMP} to establish  various 
parts formulas for the analytic Feynman integral. 
The parts formula for the analytic Feynman integral 
is now called the Cameron--Storvick theorem.

\par
On the other hand, the concept of the generalized Wiener integral  and the generalized analytic 
Feynman integral on  $C_0[0,T]$ were introduced in \cite{CPS93}, and further 
developed and used in \cite{CC17,PS91,PS95}. In \cite{CC17,CPS93,PS91,PS95}, the generalized Wiener 
integral  was defined by the Wiener integral
\[
\int_{C_0[0,T]}F (\mathcal{Z}_h(x,\cdot) )d \mathfrak{m}(x),
\]
where $\mathcal{Z}_h(x,\cdot)$ is a Gaussian path defined by the Paley--Wiener--Zygmund  (PWZ)  stochastic 
integral \cite{PWZ33,PS88} as follows:
\[
\mathcal{Z}_h(x,t) =\int_0^th(s)dx(s) \mbox{ with  } h\in L_2[0,T].
\]
 
\par
The parts formula on the function space $C_{a,b}[0,T]$,
which is a generalization of the Cameron--Storvick  theorem
was provided by Chang and Skoug in \cite{CS03} and further developed in \cite{CCS03}.
The function space $C_{a,b}[0,T]$ can be considered as the set of continuous sample paths 
of the  generalized Brownian motion process (GBMP) determined  by continuous functions 
$a(t)$ and $b(t)$ on $[0,T]$.
A GBMP on a probability space $(\Omega,\mathcal F, P)$ and a time interval $[0,T]$ 
is a Gaussian  process  $Y \equiv\{Y_t\}_{t\in [0,T]}$ such that $Y_0=0$  almost
surely,  and for any cylinder set $I_{t_1,\ldots,t_n,B}$ having the form
\[
I_{t_1,\ldots,t_n,B}
=\big\{\omega\in \Omega:(Y(t_1,\omega),\cdots,Y(t_n,\omega))\in B \big\}
\]
with a set of time moments $0= t_0 < t_1< \cdots<t_n \le T$ and a Borel set $B\subset \mathbb R^n$,
the measure $P(I_{t_1,\ldots,t_n,B})$ of $I_{t_1,\ldots,t_n,B}$ 
is equal to
\[
\begin{aligned}
&\bigg( (2\pi )^n \prod\limits_{j=1}^n\big(b(t_j)-b(t_{j-1})\big) \bigg)^{-1/2} \\
&\quad \times\int_{B}
\exp \bigg[- \frac 12 \sum\limits_{j=1}^n
   \frac {[(u_j-a(t_j))-(u_{j-1}-a(t_{j-1}))]^2}
         {b(t_j)-b(t_{j-1})} \bigg]
du_1\cdots du_n
\end{aligned}
\]
where $u_0=0$, $a(t)$ is a continuous real-valued function on  $[0,T]$,
and $b(t)$ is an increasing continuous real-valued function on $[0,T]$.
  For more details, see \cite{Yeh71,Yeh73}. Note that choosing $a(t)\equiv 0$ and $b(t)=t$ on $[0,T]$, 
one can see that the GBMP reduces a SBMP (or, Wiener process).

\par  
The aim of this paper is to provide a more general Cameron--Storvick 
theorem for the generalized analytic Feynman integral associated with  Gaussian 
paths   on the function space $C_{a,b}[0,T]$.
As an application, we apply our general Cameron--Storvick  theorem
to evaluate the generalized analytic Feynman integral of 
certain monomials in terms of PWZ  stochastic integrals.

\par 
In order to present our assertions, we assume that 
$a(t)$ is an absolutely continuous real-valued  function on $[0,T]$ 
such that  $a(0)=0$,  $a'(t)\in L^2[0,T]$, and 
\begin{equation}\label{eq:new-cc2}
\int_0^T |a'(t)|^2 d|a|(t)< +\infty,
\end{equation}
where $|a|(\cdot)$ denotes  the   total variation   function of the function $a(\cdot)$,
and  $b(t)$ is an increasing, continuously differentiable
real-valued function with $b(0)=0 $ and $b'(t) >0$ for each $t \in [0,T]$. 
We also assume familiarity with \cite{CCS03,CS03} and adopt 
the notation and terminologies of those papers. 
The basic concepts and definitions of the function space 
$(C_{a,b}[0,T],\mathcal W(C_{a,b}[0,T]),\mu)$,
which forms a complete probability space,  
the concept of the scale-invariant measurability on $C_{a,b}[0,T]$, 
the Cameron--Martin space $C_{a,b}'[0,T]$ and the PWZ 
stochastic integral on $C_{a,b}[0,T]$ may also be found in \cite{CC19,CCK15}.
In particular, we refer to the reference \cite{CCS19} 
for the definition and the properties of the Gaussian processes $\mathcal Z_k$
used in this paper. 
However, in order to propose our assertions in this paper, we shall
introduce  the following terminologies:

(i) The Hilbert space: Let
\[
L_{a,b}^2[0,T]
=\bigg\{v :\int_{0}^{T} v^2(s)db(s)<+\infty  \hbox{ and }
           \int_{0}^{T} v^2(s)d|a|(s) <+\infty \bigg\}.
\]
Then $L_{a,b}^2[0,T]$ is a separable Hilbert space with the inner product given 
by
\[
(u,v)_{a,b}=\int_0^T u(t)v(t)dm_{|a|,b}(t)\equiv \int_0^T u(t)v(t)d[b(t)+|a|(t)],
\]
where $m_{|a|,b}$ denotes  the Lebesgue--Stieltjes measure induced by $|a|(\cdot)$
and $b(\cdot)$.

(ii) The Cameron--Martin space in $C_{a,b}[0,T]$:
Let
\[
C_{a,b}'[0,T]
 =\bigg\{ w \in C_{a,b}[0,T] : w(t)=\int_0^t z(s) d b(s)
\hbox{  for some   } z \in L_{a,b}^2[0,T]  \bigg\}.
\]
Then $C_{a,b}' \equiv C_{a,b}'[0,T]$ with the inner product
\[
(w_1, w_2)_{C_{a,b}'}
=\int_0^T  Dw_1(t)  Dw_2(t)  d b(t)
\]
is a separable  Hilbert space,
where the (homeomorhic) operator 
$D: C_{a,b}'[0,T] \to L_{a,b}^2[0,T]$ is given by 
\begin{equation}\label{eq:Dt}
Dw(t)= z(t)=\frac{w'(t)}{b'(t)}.
\end{equation}

(iii) The PWZ stochastic integral:
Let $\{e_n\}_{n=1}^{\infty}$ be a complete orthonormal set in 
$(C_{a,b}'[0,T], \|\cdot\|_{C_{a,b}'})$
 such that the  $De_n$'s are  of bounded variation on $[0,T]$. 
Then for   $w\in C_{a,b}'[0,T]$  and   $x\in C_{a,b}[0,T]$, we define 
the PWZ stochastic integral $(w,x)^{\sim}$ 
as follows:
\[
(w,x)^{\sim} 
=\lim\limits_{n\to\infty}\int_0^T\sum_{j=1}^n(w,e_j)_{C_{a,b}'}De_j(t)dx(t)
\]
if the limit exists. 
For each $w\in C_{a,b}'[0,T]$, the PWZ stochastic integral $(w,x)^{\sim}$  
exists  for  a.e. $x\in C_{a,b}[0,T]$.

\setcounter{equation}{0}
\section{Gaussian processes on $C_{a,b}[0,T]$}\label{sec:GP}

\par
In order to present our Cameron--Storvick theorem on the function space $C_{a,b}[0,T]$, we
follow the exposition of   \cite{CC19,CCK15,CCS19}. 

\par
Let $C_{a,b}^*[0,T]$ be the set of  functions $k$ in $C_{a,b}'[0,T]$ such that $Dk$
is continuous except for a finite number of finite jump discontinuities and is of
bounded variation on $[0,T]$. For any $w\in C_{a,b}'[0,T]$ and $k\in C_{a,b}^*[0,T]$,
let the operation $\odot$ between $C_{a,b}'[0,T]$ and $C_{a,b}^*[0,T]$ be defined by
\[
w\odot k =D^{-1}(DwDk), \,\mbox{ i.e., } D(w\odot k)=DwDk,
\]
where $DwDk$ denotes the pointwise multiplication of the functions $Dw$
and $Dk$.  
Then $(C_{a,b}^*[0,T],\odot)$ is a commutative algebra  with the identity $b$. 

\par 
For each $t\in[0,T]$, let 
$\Phi_t(\tau)=D^{-1}\chi_{[0,t]}(\tau)=\int_0^{\tau}\chi_{[0,t]}(u)db(u)$, 
$\tau\in [0,T]$, and for  $k\in C_{a,b}'[0,T]$  with $Dk\ne 0$ $m_L$-a.e. 
on $[0,T]$ ($m_L$ denotes the Lebesgue  measure on $[0,T]$),  let  $\mathcal Z_{k}(x,t)$  
be the PWZ  stochastic integral 
\begin{equation}\label{eq:g-process}
\mathcal Z_{k}(x,t)=(k\odot  \Phi_t ,x)^{\sim}.
\end{equation}
Let
$\gamma_k(t)=\int_0^t Dk(u)da(u)$ 
and let
$\beta_k(t)=\int_0^t(Dk(u))^2db(u)$.
Then the stochastic process $\mathcal{Z}_k: C_{a,b}[0,T]\times [0,T]\to\mathbb R$ 
is Gaussian with mean function
\[
\int_{C_{a,b}[0,T]}\mathcal Z_k(x,t)d\mu(x)=\int_0^t h(u)da(u)=\gamma_k(t)
\]
and covariance function
\[
\begin{aligned}
&\int_{C_{a,b}[0,T]}\big(\mathcal Z_k(x,s)-\gamma_k(s)\big)
\big(\mathcal Z_k(x,t)-\gamma_k(t)\big)d\mu(x)\\
&=\int_0^{\min\{s,t }\{Dk(u)\}^2(u)db(u)=\beta_k(\min\{s,t\}).
\end{aligned}
\]
In addition, by \cite[Theorem 21.1]{Yeh73}, $\mathcal Z_k(\cdot,t)$ is stochastically
continuous in $t$ on $[0,T]$. If $Dk$ is of bounded variation on $[0,T]$, then, for
all $x\in C_{a,b}[0,T]$, $\mathcal Z_{k}(x,t)$ is continuous in $t$. Of course if
$k(t)\equiv b(t)$, then $\mathcal Z_b(x,t)=x(t)$, the continuous sample paths of the 
GBMP $Y$, which  consist the function space $C_{a,b}[0,T]$. 
Furthermore, if $a(t)\equiv0$ and $b(t)=t$ on $[0,T]$, then 
the function space $C_{a,b}[0,T]$   reduces to the 
classical Wiener space $C_0[0,T]$ and the Gaussian process \eqref{eq:g-process} 
with $k(t)\equiv t$ is a standard Brownian motion process.

\par
Given  any $w\in C_{a,b}'[0,T]$ and $k \in C_{a,b}^*[0,T]$, 
it follows that 
\begin{equation}\label{eq:Z-bsic-p}
(w,\mathcal{Z}_k(x,\cdot))^{\sim}
 =(w\odot k,x)^{\sim}
\end{equation}
for $\mu$-a.e $x\in C_{a,b}[0,T]$.

\par
In order to establish our Cameron--Storvick theorem for 
functionals on $C_{a,b}[0,T]$, we define a class  $\mathrm{Supp}_{C_{a,b}^*}[0,T]$ 
as follows:
\[
\mathrm{Supp}_{C_{a,b}^*}[0,T] 
=\{k\in C_{a,b}^*[0,T]:  Dk\ne 0\,\,\,\,  {m}_{L}\mbox{-a.e. on } [0,T]\}.  
\]

\begin{remark}
(i)  The  space $(\mathrm{Supp}_{C_{a,b}^*}[0,T],\odot)$ forms a monoid. 
The variance function $b(\cdot)$ of the GBMP $Y$ 
is the identity in the space $(\mathrm{Supp}_{C_{a,b}^*}[0,T]  ,  \odot)$.

(ii) Given a function $k$ in $\mathrm{Supp}_{C_{a,b}^*}[0,T]$, the process
$\mathcal Z_k$ on $C_{a,b}[0,T]\times [0,T]$
is the GBMP determined by the functions $\gamma_k$ and $\beta_k$.
\end{remark}
 
\par
For any $k\in\mathrm{Supp}_{C_{a,b}^*}[0,T]$, the   Lebesgue--Stieltjes integrals
\[
\|w\odot k\|_{C_{a,b}'}^2=\int_0^T (Dw(t))^2(Dk(t))^2db(t),
\]
and
\[
(w\odot k, a)_{C_{a,b}'} =\int_0^T Dw(t)Dk(t)Da(t)db(t)=\int_0^T Dw(t)Dk(t)da(t)
\]
exist for all $w\in C_{a,b}'[0,T]$.
Throughout the remainder of this paper,
we thus require $k$ to be in $\mathrm{Supp}_{C_{a,b}^*}[0,T]$ for the process $\mathcal Z_k$.

\setcounter{equation}{0}
\section{Parts formula for functionals in Gaussian paths}

\par
In  \cite{Cameron51}, Cameron derived an integration by parts formula for 
functionals on the Wiener space $C_0[0,T]$. 
The parts formula  involved the first  variation (a kind of G\^ateaux derivative) 
of functionals on $C_0[0,T]$. 
In this section we establish an integration by parts formula for functionals 
in Gaussian paths on the function space $C_{a,b}[0,T]$. To do this 
we first provide the definition of the first variation of  functionals on the 
function space $C_{a,b}[0,T]$.


\begin{definition}
Let $F$ be a ${\mathcal{W}}(C_{a,b}[0,T])$-measurable functional 
on $C_{a,b}[0,T]$ and let $w \in C_{a,b}[0,T]$. Then
given two functions $k_1$ and $k_2$ in $C_{a,b}[0,T]$, 
\begin{equation}\label{eq:1st}
\delta_{k_1,k_2} F(x|w)
\equiv \delta  F(\mathcal Z_{k_1}(x,\cdot)|\mathcal Z_{k_2}(w,\cdot))
=\frac{\partial}{\partial\alpha} 
F(\mathcal Z_{k_1}(x,\cdot)+\alpha \mathcal Z_{k_2}(w,\cdot)) \bigg|_{\alpha=0}
\end{equation}
(if it exists) is called the first variation of $F$ in the direction $w$.
\end{definition}

\begin{remark} \label{re:2020oct}
Setting $k_1=k_2 \equiv b$ on $[0,T]$, our definition of the first variation reduces
to the first variation studied in \cite{CCS03,CS03}. That is,
\[
\delta_{b,b} F(x|w)=\delta F(x|w).
\]
\end{remark}

\par
Let $\mathcal Z_k$  be the  Gaussian process given by \eqref{eq:g-process} 
on $C_{a,b}[0,T]\times[0,T]$. We define the  $\mathcal Z_k$-function space
integral (namely, the function space integral associated with  the Gaussian 
paths $\mathcal Z_k(x,\cdot)$) for functionals $F$ on $C_{a,b}[0,T]$ by the 
formula
\[
E_{x}[F(\mathcal Z_k(x,\cdot))]
=\int_{C_{a,b}[0,T]}F (\mathcal Z_k(x,\cdot))d\mu(x)
\]
whenever the integral exists.

\par
In order to establish an integration  by parts formula 
for the function space integral associated with Gaussian paths on $C_{a,b}[0,T]$,
we need a translation theorem for the function space integral.
The following translation theorem is due to Chang and Choi \cite{CC19}.

\begin{theorem}\label{lemma:translation-Z}
Let  $k_1$ be a  function in $\mathrm{Supp}_{C_{a,b}^*}[0,T]$  and
let $F$ be a functional on $C_{a,b}[0,T]$ such that $F(\mathcal Z_{k_1}(x,\cdot))$ 
 is $\mu$-integrable over $C_{a,b}[0,T]$.
Then for any $\theta \in C_{a,b}'[0,T]$ and  $k_2 \in \mathrm{Supp}_{C_{a,b}^*}[0,T]$,
\begin{equation}\label{eq:tt001-v2}
\begin{aligned}
&E_{x}\big[F(\mathcal{Z}_{k_1}(x,\cdot)+\mathcal{Z}_{k_2}(\theta\odot k_1,\cdot))\big]\\
&=\exp\bigg[ -\frac{1}{2}\|\theta \odot k_2\|_{C_{a,b}'}^2
-(\theta \odot k_2,a)_{C_{a,b}'}\bigg]\\
& \quad \times
E_{x}\Big[ F(\mathcal{Z}_{k_1}(x,\cdot))
\exp\big[(\theta,\mathcal{Z}_{k_2}(x,\cdot))^{\sim}\big]\Big].
\end{aligned}
\end{equation}
\end{theorem}

\par
We are now ready to present our integration by parts formula for functionals
in Gaussian paths on $C_{a,b}[0,T]$.

\begin{theorem}\label{byparts-step1}
Let  $k_1$ and $k_2$ be functions in $\mathrm{Supp}_{C_{a,b}^*}[0,T]$,  
let $\theta$ be a function in $C_{a,b}'[0,T]$,
and let $F$ be a functional on $C_{a,b}[0,T]$ such that
$F(\mathcal Z_{k_1}(x,\cdot))$ is $\mu$-integrable over $C_{a,b}[0,T]$.
Furthermore assume that  
\begin{equation}\label{step1-condition}
E_{x}\big[\big|\delta  F(\mathcal Z_{k_1}(x,\cdot)
|\mathcal Z_{k_2}(\theta\odot k_1,\cdot))\big|\big]<+\infty.
\end{equation}
Then 
\begin{equation}\label{eq:byparts-step1-show}
\begin{aligned}
&E_{x}\big[\delta  F(\mathcal Z_{k_1}(x,\cdot)|\mathcal Z_{k_2}(\theta\odot k_1,\cdot))\big]\\
&=E_{x}\big[(\theta,\mathcal Z_{k_2}(x,\cdot))^{\sim} F(\mathcal Z_{k_1}(x,\cdot))\big]
-(\theta\odot k_2,a)_{C_{a,b}'}E_{x}\big[F(\mathcal Z_{k_1}(x,\cdot))\big].
\end{aligned}
\end{equation}
\end{theorem}
\begin{proof}
By using  \eqref{eq:1st}  and \eqref{eq:tt001-v2}, it follows that
\begin{equation}\label{setp1-evlu}
\begin{aligned}
&E_{x}\big[\delta  F(\mathcal Z_{k_1}(x,\cdot)|\mathcal Z_{k_2}(\theta\odot k_1,\cdot))\big]\\
&=E_{x}\bigg[ \frac{\partial}{\partial \alpha}F(\mathcal Z_{k_1}(x,\cdot)
+\alpha\mathcal Z_{k_2}(\theta\odot k_1,\cdot)) \bigg |_{\alpha=0} \bigg]\\
&=\frac{\partial}{\partial \alpha}\Big(E_{x}\big[
 F(\mathcal Z_{k_1}(x,\cdot)+\mathcal Z_{\alpha k_2}(\theta\odot k_1,\cdot)) \big]\Big) 
\bigg |_{\alpha=0}\\
&=\frac{\partial}{\partial \alpha}
\bigg ( \exp\bigg[ -\frac{\alpha^2}{2}\|\theta \odot k_2\|_{C_{a,b}'} 
-\alpha(\theta\odot k_2,a)_{C_{a,b}'}\bigg]\\
&\qquad\qquad\times
E_{x}\Big[F(\mathcal Z_{k_1}(x,\cdot))
\exp\big[ \alpha(\theta,\mathcal Z_{k_2}(x,\cdot))^{\sim}\big]
\Big] \bigg ) \bigg|_{\alpha=0}\\
&=E_{x}\big[(\theta,\mathcal Z_{k_2}(x,\cdot))^{\sim} F(\mathcal Z_{k_1}(x,\cdot))\big] 
-(\theta\odot k_2,a)_{C_{a,b}'}E_{x}\big[F(\mathcal Z_{k_1}(x,\cdot))\big].
\end{aligned}
\end{equation}
The second equality of  \eqref{setp1-evlu}  
follows from \eqref{step1-condition} and Theorem 2.27 in \cite{folland}.
\qed\end{proof}

\setcounter{equation}{0}
\section{Cameron--Storvick  theorem for  the  generalized analytic 
Feynman integral associated with Gaussian paths}\label{sec:translation-g}
 
\par
In this section, we establish the Cameron--Storvick theorem for the generalized 
analytic Feynman integral of functionals $F$ on 
the function space $C_{a,b}[0,T]$. We begin this section with the definition of 
the generalized analytic Feynman integral associated with Gaussian 
process $\mathcal Z_k$ ($\mathcal Z_k$-Feynman integral)   on $C_{a,b}[0,T]$.
 
\par
Throughout the remainder of this paper, let  $\mathbb C_+$ and $\mathbb{\widetilde C}_+$
denote the set of   complex numbers with positive real part,  and  non-zero
complex numbers with nonnegative real part, respectively. For each $\lambda \in \mathbb C$,
$\lambda^{1/2}$ denotes the principal square root of $\lambda$; i.e., $\lambda^{1/2}$ is
always chosen to have nonnegative real part, so that  $\lambda^{-1/2}=(\lambda^{-1})^{1/2}$
is  in $\mathbb C_+$ for all $\lambda \in \widetilde{\mathbb C}_+$.

\begin{definition} \label{def:Faynman}
Given a function $k\in \mathrm{Supp}_{C_{a,b}^*}[0,T]$,
let $\mathcal Z_k$  be the  Gaussian process given by \eqref{eq:g-process} 
and let $F$ be a $\mathbb C$-valued scale-invariant measurable 
functional on $C_{a,b}[0,T]$ such that the generalized $\mathcal Z_k$-function 
space  integral (namely, the function space integral associated with the Gaussian 
paths $\mathcal Z_k(x,\cdot)$) 
\[
J_F(\mathcal Z_k;\lambda)
=  E_{x} [F (\lambda^{-1/2}\mathcal Z_k(x,\cdot))]
\]
exists and is finite for all $\lambda>0$. If there exists  a 
function $J_F^*(\mathcal Z_k;\lambda)$ analytic on $\mathbb C_+$    such that 
$J_F^*(\mathcal Z_k;\lambda)=J_F(\mathcal Z_k;\lambda)$  for all $\lambda\in (0,+\infty)$, 
then $J_F^*(\mathcal Z_k;\lambda)$ is defined to be the  analytic 
$\mathcal Z_k$-function space  integral (namely,  the analytic function space 
integral associated with the Gaussian paths 
$\mathcal Z_k(x,\cdot)$) of $F$  over $C_{a,b}[0,T]$ 
with parameter $\lambda$, and for $\lambda\in\mathbb C_+$   we write
\[
E_{x}^{\mathrm{an}_{\lambda}}[F(\mathcal Z_k(x,\cdot))]
\equiv \int_{C_{a,b}[0,T]}^{\mathrm{an}_{\lambda}}F(\mathcal Z_k(x,\cdot))d\mu(x)
 = J_F^*(\mathcal Z_k;\lambda).
\]

\par
Let $q$ be a non-zero real number and 
let $F$ be a  scale-invariant measurable  functional whose
analytic $\mathcal Z_k$-function space integral, 
$E_{x}^{\mathrm{an}_{\lambda}}[F(\mathcal Z_k(x,\cdot))]$, 
exists for  all $\lambda$ in $\mathbb C_+$. If the following limit exists, 
we call it the generalized  analytic $\mathcal Z_k$-Feynman  
integral  of $F$ with parameter $q$, and we write
\begin{equation}\label{eq:Feynman-add400}
E_{x}^{\mathrm{anf}_{q}}[F(\mathcal Z_k(x,\cdot))]
\equiv \int_{C_{a,b}[0,T]}^{\mathrm{anf}_{q}}F(\mathcal Z_k(x,\cdot))d\mu(x) 
=\lim\limits_{\substack{\lambda\to -iq \\  \lambda\in \mathbb C_+}}
E_{x}^{\mathrm{an}_{\lambda}}[F(\mathcal Z_k(x,\cdot))].
\end{equation}
\end{definition}

\par
We are now ready to establish a Cameron--Storvick type theorem for 
our generalized analytic Feynman integral.
It will be helpful to establish the following  lemma  before giving 
the main theorem.

\begin{lemma}\label{byparts-step2}
Let $k_1$, $k_2$, $\theta$,  and $F$ be as in Theorem \ref{byparts-step1}.
For each $\rho>0$, assume that $F(\rho \mathcal Z_{k_1}(x,\cdot))$ is $\mu$-integrable 
over $C_{a,b}[0,T]$. 
Furthermore assume that for each $\rho>0$,
\[
E_{x}\big[\big|\delta  F(\rho \mathcal Z_{k_1}(x,\cdot)
|\rho\mathcal Z_{k_2}(\theta\odot k_1,\cdot))\big|\big]<+\infty.
\]
Then 
\begin{equation}\label{eq:byparts-step2-show}
\begin{aligned}
&E_{x}\big[\delta  F(\rho \mathcal Z_{k_1}(x,\cdot)
|\rho\mathcal Z_{k_2}(\theta\odot k_1,\cdot))\big]\\
&=E_{x}\big[(\theta,\mathcal Z_{k_2}(x,\cdot))^{\sim} 
F(\rho \mathcal Z_{k_1}(x,\cdot))\big]
-(\theta\odot k_2,a)_{C_{a,b}'}
E_{x}\big[F(\rho \mathcal Z_{k_1}(x,\cdot))\big].
\end{aligned}
\end{equation}
\end{lemma}
\begin{proof}
Let $G(x)=F(\rho x)$. Then 
\[
G(\mathcal Z_{k_1}(x,\cdot)+\alpha \mathcal Z_{k_2}(w,\cdot))
=F(\rho  \mathcal Z_{k_1}(x,\cdot)+\rho \alpha \mathcal Z_{k_2}(w,\cdot))
\] 
and 
\[
\frac{\partial}{\partial \alpha}
G(\mathcal Z_{k_1}(x,\cdot)+\alpha \mathcal Z_{k_2}(w,\cdot))\bigg|_{\alpha=0}
=\frac{\partial}{\partial \alpha}
F(\rho  \mathcal Z_{k_1}(x,\cdot)+\rho \alpha \mathcal Z_{k_2}(w,\cdot)) \bigg|_{\alpha=0}.
\]
Thus  $\delta  G(\mathcal Z_{k_1}(x,\cdot)|\mathcal Z_{k_2}(\theta\odot k_1,\cdot))
=\delta  F(\rho \mathcal Z_{k_1}(x,\cdot)|\rho\mathcal Z_{k_2}(\theta\odot k_1,\cdot))$. 
Hence by  equation \eqref{eq:byparts-step1-show} 
with $F$ replaced with $G$, we have
\[
\begin{aligned}
&E_{x}\big[\delta F(\rho \mathcal Z_{k_1}(x,\cdot)
|\rho\mathcal Z_{k_2}(\theta\odot k_1,\cdot))\big]\\
&=E_{x}\big[\delta G(\mathcal Z_{k_1}(x,\cdot)|\mathcal Z_{k_2}(\theta\odot k_1,\cdot))\big]\\
&=E_{x}\big[(\theta,\mathcal Z_{k_2}(x,\cdot))^{\sim} G(\mathcal Z_{k_1}(x,\cdot))\big] 
-(\theta\odot k_2,a)_{C_{a,b}'}E_{x}\big[G(\mathcal Z_{k_1}(x,\cdot))\big]\\
&=E_{x}\big[(\theta,\mathcal Z_{k_2}(x,\cdot))^{\sim} F(\rho \mathcal Z_{k_1}(x,\cdot))\big]
-(\theta\odot k_2,a)_{C_{a,b}'}E_{x}\big[F(\rho\mathcal Z_{k_1}(x,\cdot))\big]
\end{aligned}
\]
which establishes \eqref{eq:byparts-step2-show}.
\qed\end{proof}

\par
Next we provide the Cameron--Storvick theorem for the generalized analytic
$\mathcal Z_k$-Feynman integral on  the function space $C_{a,b}[0,T]$.

\begin{theorem}\label{thm:CS-real}
Let $k_1$, $k_2$,   $\theta$,  and $F$ be as in Lemma \ref{byparts-step2}.
Then if any two of the three generalized analytic Feynman integrals in the 
following equation exist, then the third one also exists, and equality holds:
\[
\begin{aligned}
&E_{x}^{\mathrm{anf}_q}\big[\delta F( \mathcal Z_{k_1} (x,\cdot)
| \mathcal Z_{k_2}(\theta\odot k_1,\cdot))\big]\\
&=-iq E_{x}^{\mathrm{anf}_q}\big[
(\theta,\mathcal Z_{k_2}(x,\cdot))^{\sim} F( \mathcal Z_{k_1}(x,\cdot))\big]\\
&\quad
-(-iq)^{1/2}(\theta\odot k_2,a)_{C_{a,b}'}
E_{x}^{\mathrm{anf}_q}\big[F( \mathcal Z_{k_1}(x,\cdot))\big].
\end{aligned}
\]
\end{theorem}
\begin{proof}
Given  $\rho>0$  and $\theta \in C_{a,b}'[0,T]$, let $\theta_{\rho}=\frac{1}{\rho} \theta$.
Then $\theta_{\rho}$ is a function  in $C_{a,b}'[0,T]$, and 
$\theta \odot k_1= \rho \theta_{\rho}\odot k_1$. 
By  equation \eqref{eq:byparts-step2-show} with $\theta$ replaced with $\theta_{\rho}$, 
\begin{equation}\label{setp3-evlu}
\begin{aligned}
&E_{x}\big[\delta  F (\rho \mathcal Z_{k_1}(x,\cdot) 
|\mathcal Z_{k_2}(\theta \odot k_1,\cdot))\big]\\
&=E_{x}\big[ \delta  F(\rho \mathcal Z_{k_1}(x,\cdot) 
|\rho \mathcal Z_{k_2}(\theta_{\rho} \odot k_1,\cdot)) \big]\\
&=E_{x}\big[( \theta_{\rho},\mathcal Z_{k_2 }(x,\cdot))^{\sim} 
F(\rho \mathcal Z_{k_1}(x,\cdot)) \big] 
 -( \theta_{\rho}\odot k_2 ,a)_{C_{a,b}'}E_{x}\big[F(\rho\mathcal Z_{k_1}(x,\cdot))\big] \\
&=\rho^{-2}E_{x}\big[(\theta,\rho\mathcal Z_{k_2 }(x,\cdot))^{\sim}
 F(\rho \mathcal Z_{k_1}(x,\cdot))\big]\\
&\quad
-\rho^{-1}( \theta\odot k_2 ,a)_{C_{a,b}'}E_{x}\big[F(\rho\mathcal Z_{k_1}(x,\cdot))\big].
\end{aligned}
\end{equation}
Now let $\rho=\lambda^{-1/2}$. Then equation \eqref{setp3-evlu} becomes
\begin{equation}\label{setp3-evlu2}
\begin{aligned}
&E_{x}\big[ \delta  
F(\lambda^{-1/2} \mathcal Z_{k_1}(x,\cdot) |\mathcal Z_{k_2}(w,\cdot))\big]\\
&=\lambda E_{x}\big[ (\theta,\lambda^{-1/2}\mathcal Z_{k_2 }(x,\cdot))^{\sim} 
F(\lambda^{-1/2} \mathcal Z_{k_1}(x,\cdot))\big]\\
&\quad
-\lambda^{1/2}( \theta\odot k_2 ,a)_{C_{a,b}'}
E_{x}\big[ F(\lambda^{-1/2}\mathcal Z_{k_1}(x,\cdot))\big].
\end{aligned}
\end{equation}
Since $\rho>0$ was arbitrary, we have that equation \eqref{setp3-evlu2} holds 
for all $\lambda>0$. We now use Definition \ref{def:Faynman} to obtain our desired 
conclusions.
\qed\end{proof}

\begin{corollary} 
Under the assumptions as given in Theorem \ref{thm:CS-real},
it follows that 
if any two of the three generalized analytic Feynman integrals 
in the following equation exist, then the third one also exists, 
and equality holds:
\begin{equation}\label{eq:byparts-step3-show-ex}
\begin{aligned}
&E_{x}^{\mathrm{anf}_q}\big[(\theta,\mathcal Z_{k_2}(x,\cdot))^{\sim} 
F( \mathcal Z_{k_1}(x,\cdot))\big]\\
&=\frac{i}{q}E_{x}^{\mathrm{anf}_q}\big[
\delta F( \mathcal Z_{k_1} (x,\cdot)| \mathcal Z_{k_2}(\theta\odot k_1,\cdot))\big]\\
&\quad+(-iq)^{-1/2}
(\theta\odot k_2,a)_{C_{a,b}'}
E_{x}^{\mathrm{anf}_q}\big[F( \mathcal Z_{k_1}(x,\cdot))\big].
\end{aligned}
\end{equation}
\end{corollary}

\begin{remark}\label{remark-old}
As commented  in Section \ref{sec:GP} above,
if $k\equiv b$ on $[0,T]$, then  $\mathcal Z_b(x,t)=x(t)$ for 
each $x\in C_{a,b}[0,T]$. In this case the generalized analytic 
$\mathcal Z_b$-Feynman integral $E_{x}^{\mathrm{anf}_q}[F(\mathcal Z_b(x,\cdot))]$ 
agrees  with the previous definition  of the generalized analytic Feynman 
integral $E_x^{\mathrm{anf}_q}[F(x)]$, see  \cite{CCS03,CS03}. 
\end{remark}
\par
In view of Remarks \ref{re:2020oct} and  \ref{remark-old}, we have the following corollary.

\begin{corollary}[\cite{CCS03}]
Let $\theta$ be a function in $C_{a,b}'[0,T]$,
and let $F$ be a functional on $C_{a,b}[0,T]$ such that
for each $\rho>0$, $F(\rho x)$ is $\mu$-integrable over $C_{a,b}[0,T]$. 
Furthermore assume that for each $\rho>0$, 
\[
E_{x}\big[\big|\delta  F(\rho \mathcal Z_b(x,\cdot) |\rho\mathcal Z_b(\theta\odot b,\cdot))\big|\big] 
\equiv  E_{x}\big[\big|\delta  F(\rho x |\rho \theta)\big|\big]<+\infty.
\]
Then if any two of the three generalized analytic Feynman integrals 
in the following equation exist, 
then the third one also exists, and equality holds:
\[
E_{x}^{\mathrm{anf}_q}[\delta F( x| \theta)] 
=-iq E_{x}^{\mathrm{anf}_q}[(\theta,x)^{\sim} F(x)]
-(-iq)^{1/2}(\theta,a)_{C_{a,b}'}E_{x}^{\mathrm{anf}_q}[F(x)].
\]
\end{corollary}

\par
The formulas and results in this paper are more complicated than the
corresponding formulas and results in \cite{Cameron51,CS91,PS98-PanAmer,PSS98-RCMP}
because the Gaussian process used in this paper is neither   centered nor stationary in time.
However, by choosing $a(t)\equiv 0$ and $b(t)=t$ on $[0,T]$, the function space $C_{a,b}[0,T]$
reduces to the Wiener space $C_0[0,T]$, and so the
expected results on  $C_0[0,T]$ are immediate corollaries of the results in this paper.

\setcounter{equation}{0}
\section{Generalized analytic Feynman integral of  
monomials  in terms of PWZ stochastic integrals}

\par
When we evaluate the following generalized analytic Feynman integral
\begin{equation}\label{eq:exam2020}
E_x^{\mathrm{anf}_q}\bigg[\prod_{j=1}^{m } (\theta\odot k_j,x)^{\sim}\bigg]
\end{equation}
we might not be able to use the change of variables theorem of the usual measure theory,
because the set of Gaussian random variables $(\theta\odot k_j,x)^{\sim}$, $j=1,\ldots,m$,
are generally not independent.
In this case, to  apply  the change of variables theorem  for the calculation 
of \eqref{eq:exam2020}, we might apply the Gram--Schmidt process for the set of 
functions $\{\theta\odot k_1,  \ldots,\theta\odot k_m\}$.
 
\par
Using equation \eqref{eq:byparts-step3-show-ex}, we indeed see that  
the generalized Feynman integral of  functionals having the 
form \eqref{eq:exam2020} can be calculated very explicitly. In this section  
we present interesting  examples to which equation \eqref{eq:byparts-step3-show-ex}
can be applied.

\begin{example}\label{example01}
Let  $k_1$ and $k_2$ be functions in $\mathrm{Supp}_{C_{a,b}^*}[0,T]$,
and given a function $\theta$  in $C_{a,b}'[0,T]$,  set $F(x)=(\theta,x)^{\sim}$ 
for $x\in C_{a,b}[0,T]$. Then using equations \eqref{eq:1st} 
and \eqref{eq:Z-bsic-p}, it follows that   for  any $w$  in $C_{a,b}'[0,T]$,
\[
\begin{aligned}
\delta  F(\mathcal Z_{k_1}(x,\cdot)|\mathcal Z_{k_2}(w,\cdot))
&=\frac{\partial}{\partial\alpha}  \big\{(\theta, \mathcal Z_{k_1}(x,\cdot))^{\sim}
+\alpha (\theta,\mathcal Z_{k_2}(w,\cdot))^{\sim}\big\} \bigg|_{\alpha=0}\\
&=(\theta,\mathcal Z_{k_2}(w,\cdot))^{\sim} \\
&=(\theta\odot k_2,w)^{\sim} \\
&=(\theta\odot k_2,w)_{C_{a,b}'} .
\end{aligned}
\]
From this, we see that 
\begin{equation}\label{eq:ex102}
\delta  F(\mathcal Z_{k_1}(x,\cdot)|\mathcal Z_{k_2}(\theta\odot k_1,\cdot))
=(\theta\odot k_2, \theta \odot k_1)_{C_{a,b}'}.
\end{equation}
Also using \eqref{eq:Feynman-add400}, it follows that 
\begin{equation}\label{eq:ex103}
\begin{aligned}
E_{x}^{\mathrm{anf}_q}\big[F(\mathcal Z_{k_1}(x,\cdot))\big]
&=E_{x}^{\mathrm{anf}_q}\big[(\theta,\mathcal Z_{k_1}(x,\cdot))^{\sim}\big] \\
&=E_{x}^{\mathrm{anf}_q}\big[(\theta\odot k_1,x)^{\sim}\big]
=(-iq)^{-1/2}(\theta\odot k_1,a)_{C_{a,b}'}.
\end{aligned}
\end{equation} 

\par
Next using equations \eqref{eq:byparts-step3-show-ex}, \eqref{eq:ex102}, and \eqref{eq:ex103},
we obtain the formula
\begin{equation}\label{eq:ex104}
\begin{aligned}
&E_x^{\mathrm{anf}_q}\big[(\theta\odot k_2,x)^{\sim}
(\theta\odot k_1,x))^{\sim}\big]\\
&\equiv  
\int_{C_{a,b}[0,T]}^{\mathrm{anf}_q}(\theta,\mathcal Z_{k_2}(x,\cdot))^{\sim}
(\theta,\mathcal Z_{k_1}(x,\cdot))^{\sim}   d\mu(x)\\
&= E_{x}^{\mathrm{anf}_q}\big[(\theta,\mathcal Z_{k_2}(x,\cdot))^{\sim}
F(\mathcal Z_{k_1}(x,\cdot))\big]\\
&=\frac{i}{q}E_{x}^{\mathrm{anf}_q}\big[ \delta F( \mathcal Z_{k_1} (x,\cdot)
| \mathcal Z_{k_2}(\theta \odot k_1,\cdot))\big]\\
&\quad+(-iq)^{-1/2}
(\theta\odot k_2,a)_{C_{a,b}'}E_{x}^{\mathrm{anf}_q}\big[F(\mathcal Z_{k_1}(x,\cdot))\big]\\
&=\frac{i}{q}(\theta\odot k_2, \theta \odot k_1)_{C_{a,b}'}
+\frac{i}{q}(\theta\odot k_2,a)_{C_{a,b}'}(\theta \odot k_1,a)_{C_{a,b}'}.
\end{aligned}
\end{equation}
\end{example}

\par
In our next example, for any positive integer $m\in \{3,4,\ldots\}$, we obtain a recurrence
relation for the generalized analytic Feynman integral
\[
E_x^{\mathrm{anf}_q}\bigg[\prod_{j=1}^m(\theta\odot k_j,x)^{\sim}\bigg].
\]

\begin{example}
For a positive integer $m \ge 3$,
let  $\{k_1, \ldots,k_{m-1}, k_m\}$ be a finite  set  
of functions in  $\mathrm{Supp}_{C_{a,b}^*}[0,T]$,
and given a function $\theta\in C_{a,b}'[0,T]$, set 
\[
\begin{aligned}
F(x)
=\prod_{j=1}^{m-1}(\theta\odot k_j ,x)^{\sim}
=\prod_{j=1}^{m-1}(\theta\odot k_j ,\mathcal Z_b(x,\cdot))^{\sim}.
\end{aligned}
\]
First, using equation \eqref{eq:1st}, it follows that for all $w\in C_{a,b}'[0,T]$,
\[
\begin{aligned}
\delta F(&x|\mathcal Z_{k_m}(w,\cdot))
=\delta F(\mathcal Z_{b}(x,\cdot)|\mathcal Z_{k_m}(w,\cdot))\\
&=\frac{\partial}{\partial \alpha} 
\prod_{j=1}^{m-1}\big\{(\theta\odot k_j ,\mathcal Z_{b}(x,\cdot) )^{\sim} 
+ \alpha(\theta\odot k_j ,  \mathcal Z_{k_m}(w,\cdot))^{\sim}\big\} \bigg|_{\alpha=0}\\
&=\sum_{l=1}^{m-1}\bigg(\prod\limits_{\begin{subarray}{1}j=1\\j\ne l\end{subarray}}^{m-1}
(\theta\odot k_j ,\mathcal Z_{b}(x,\cdot) )^{\sim}\bigg) 
(\theta\odot k_l ,\mathcal Z_{k_m}(w,\cdot) )^{\sim}. 
\end{aligned}
\]
Then, in particular, it follows that
\begin{equation}\label{eq:ex105}
\begin{aligned}
 \delta F(&\mathcal Z_{b}(x,\cdot) |\mathcal Z_{k_m}(\theta \odot b,\cdot))   
 =\delta F(\mathcal Z_{b}(x,\cdot)|\mathcal Z_{k_m}(\theta,\cdot))\\
&=\sum_{l=1}^{m-1} \bigg(\prod\limits_{\begin{subarray}{1}j=1\\j\ne l\end{subarray}}^{m-1}
(\theta\odot k_j ,\mathcal Z_{b}(x,\cdot) )^{\sim} \bigg)
(\theta\odot k_l ,\mathcal Z_{k_m}(\theta,\cdot) )^{\sim} \\
&=\sum_{l=1}^{m-1}\bigg(\prod\limits_{\begin{subarray}{1}j=1\\j\ne l\end{subarray}}^{m-1}
(\theta\odot k_j ,\mathcal Z_{b}(x,\cdot) )^{\sim} \bigg)
(\theta\odot k_l, \theta\odot k_m )_{C_{a,b}'}. 
\end{aligned}
\end{equation}

\par
Next, using   \eqref{eq:byparts-step3-show-ex} with $k_1$ and $k_2$ 
replaced with $b$ and  $k_m$, respectively, and with 
$F(x)=\prod_{j=1}^{m-1}(\theta\odot k_j,x)^{\sim}$,
\eqref{eq:ex105}, and   the equation $\mathcal Z_{b}(x,\cdot)=x$, 
it follows that 
\begin{equation}\label{eq:ex106}
\begin{aligned}
& E_x^{\mathrm{anf}_q}\bigg[\prod_{j=1}^m(\theta\odot k_j,x)^{\sim}\bigg]\\
&=E_{x}^{\mathrm{anf}_q}\bigg[
(\theta,\mathcal Z_{k_m}(x,\cdot))^{\sim}
\prod_{j=1}^{m-1}(\theta\odot k_j,\mathcal Z_{b}(x,\cdot))^{\sim}\bigg]\\
&= E_{x}^{\mathrm{anf}_q}\big[(\theta,\mathcal Z_{k_m}(x,\cdot))^{\sim}
F(\mathcal Z_b(x,\cdot))\big]\\
&=\frac{i}{q}E_{x}^{\mathrm{anf}_q}\big[ 
\delta F( \mathcal Z_{b} (x,\cdot)| \mathcal Z_{k_m}(\theta\odot b,\cdot))\big]\\
&\quad+(-iq)^{-1/2}
(\theta\odot k_m,a)_{C_{a,b}'}E_{x}^{\mathrm{anf}_q}\big[
F( \mathcal Z_{b}(x,\cdot))\big]\\
&=\frac{i}{q}\sum_{l=1}^{m-1} (\theta\odot k_l, \theta\odot k_m )_{C_{a,b}'}
E_x^{\mathrm{anf}_q}\bigg[\prod\limits_{\begin{subarray}{1}j=1\\j\ne l\end{subarray}}^{m-1}
(\theta\odot k_j,x)^{\sim}\bigg]\\
&\quad+(-iq)^{-1/2}
(\theta \odot k_m,a)_{C_{a,b}'}E_x^{\mathrm{anf}_q}\bigg[
\prod_{j=1}^{m-1}(\theta\odot k_j,x)^{\sim}\bigg].
\end{aligned}
\end{equation}
\end{example}

\begin{remark}
Letting $m=3$ in equation \eqref{eq:ex106} and applying
equations \eqref{eq:ex103}  and \eqref{eq:ex104} allows us to 
easily and completely calculate the 
generalized Feynman integral 
\[
\begin{aligned}
&E_x^{\mathrm{anf}_q}\big[
(\theta\odot k_1,x)^{\sim}
(\theta\odot k_2x)^{\sim}
(\theta\odot k_3,x)^{\sim}\big]\\
&=E_x^{\mathrm{anf}_q}
\big[
(\theta,\mathcal Z_{k_1}(x,\cdot))^{\sim}
(\theta,\mathcal Z_{k_2}(x,\cdot))^{\sim}
(\theta,\mathcal Z_{k_3}(x,\cdot))^{\sim} \big].
\end{aligned}
\]
Then setting $m=4$ in equation  \eqref{eq:ex106} allows us 
to completely evaluate the  generalized Feynman integral 
\[
\begin{aligned}
&E_x^{\mathrm{anf}_q}
\big[
(\theta\odot k_1,x)^{\sim}
(\theta\odot k_2x)^{\sim}
(\theta\odot k_3,x)^{\sim}
(\theta\odot k_4,x)^{\sim} \big]\\
&=E_x^{\mathrm{anf}_q}
\big[
(\theta,\mathcal Z_{k_1}(x,\cdot))^{\sim}
(\theta,\mathcal Z_{k_2}(x,\cdot))^{\sim}
(\theta,\mathcal Z_{k_3}(x,\cdot))^{\sim}
(\theta,\mathcal Z_{k_4}(x,\cdot))^{\sim} \big],
\end{aligned}
\]
since we already have complete evaluation formulas for 
\[
E_x^{\mathrm{anf}_q}\bigg[
\prod_{j=1}^{l}(\theta\odot k_j,x)^{\sim}\bigg],\,\, l=1,2, 3. 
\]
Then we can evaluate 
\[
E_x^{\mathrm{anf}_q}\bigg[\prod_{j=1}^{5}(\theta\odot k_j ,x)^{\sim}\bigg], 
\]
since we  have already   evaluated  
\[
E_x^{\mathrm{anf}_q}\bigg[
\prod_{j=1}^{l}(\theta\odot k_j ,x)^{\sim}\bigg]
\]
for  $l=1,2,3$ and $4$; etc.
\end{remark}


 \end{document}